\documentclass[a4paper,12pt,twoside]{article}

\usepackage{doi}

\usepackage[dvipsnames]{xcolor} 

\usepackage{hyperref} 
\hypersetup{breaklinks, colorlinks,
    linkcolor = {Blue},
    citecolor = {Blue},
    urlcolor  = {Blue}
}

\usepackage{float}

\usepackage{amsmath,amssymb}
\usepackage{mathtools} 
\usepackage{alphabeta} 

\usepackage{amsthm, 
}

\usepackage[numbers,sort&compress]{natbib}
\usepackage[text={6.5in,9.5in},centering]{geometry}

\usepackage[inline,shortlabels]{enumitem} 
\renewlist{enumerate}{enumerate}{2}
\setlist[enumerate,1]{label=\textup{(\alph*)},
ref={\alph*}, align=left, labelsep=0.5ex, leftmargin=*}
\setlist[enumerate,2]{label=\textup{({\roman*})},
ref={\roman*}, align=right, labelsep*=1ex, widest={(ii)},  
leftmargin=5.4ex}

\DeclareSymbolFont{bbold}{U}{bbold}{m}{n}
\DeclareMathSymbol{\bbS}{\mathord}{bbold}{83}      

\theoremstyle{plain}
\newtheorem{theorem}{Theorem}[section]
\newtheorem{corollary}{Corollary}[section]
\newtheorem{lemma}{Lemma}[section]

\newtheoremstyle{boldremex}
    {\dimexpr\topsep/2\relax} 
    {\dimexpr\topsep/2\relax} 
    {}          
    {}          
    {\bfseries} 
    {.}         
    {.5em}      
    {}          
    
\theoremstyle{boldremex}
\newtheorem{remark}{Remark}[section]

\makeatletter
\renewenvironment{proof}[1][\proofname]{%
   \par\pushQED{\qed}\normalfont%
   \topsep6\p@\@plus6\p@\relax
   \trivlist\item[\hskip\labelsep\bfseries#1\@addpunct{.}]%
   \ignorespaces
}{%
   \popQED\endtrivlist\@endpefalse
}
\makeatother

\newcommand{\conc}[2]{{{{\mathrm{Conc}}}(#1\,;\,#2)}}

\newcommand{\CC}{\mathbb{C}}
\newcommand{\RR}{\mathbb{R}}
\newcommand{\NN}{\mathbb{N}}
\newcommand{\ZZ}{\mathbb{Z}}
\newcommand{\Zpl}{\ZZ_+}

\newcommand{\Comp}{\mathrm{Comp}}
\newcommand{\Part}{\mathrm{Part}}

\newcommand{\set}[1]{\underline{#1}}

\newcommand{\dd}{{\mathrm{d}}}
\newcommand{\ee}{{\mathrm{e}}}

\newcommand{\pii}{{\text{\pi}}}
 
\newcommand{\abs}[1]{|#1|}                   

\newcommand{\ABS}[1]{\Bigl|#1\Bigr|}           

\newcommand{\Be}{\mathrm{Be}}

\newcommand{\bigast}{\mathop{\hbox{\Large$\ast$}}}

\newcommand{\bfcdot}{{\boldsymbol{\cdot}}}

\newcommand{\card}[1]{|#1|}                 

\newcommand{\dirac}{\delta}                 

\newcommand{\dtv}[2]{d_{\mathrm{TV}}(#1,#2)}

\newcommand{\EE}{\mathrm{E}\,} 

\newcommand{\floor}[1]{{\lfloor #1 \rfloor}} 

\newcommand{\gnhaf}{\mathrm{gnhaf}}

\newcommand{\haf}{\mathrm{haf}}

\newcommand{\Subs}[2]{\bbS(#1,#2)} 

\newcommand{\Var}{\mathrm{Var}\,} 

\newcommand{\binomial}[2]{\genfrac{(}{)}{0pt}{}{#1}{#2}}

\scrollmode

\begin{document}
\title{Smoothness and L\'{e}vy concentration function inequalities
for distributions of random diagonal sums}
\author{
Bero Roos\footnote{
Postal address:
FB IV -- Dept.\ of Mathematics, 
University of Trier, 
54286 Trier, Germany. 
E-mail: \texttt{bero.roos@uni-trier.de}}\medskip\\
University of Trier 
\date{
}
}
\maketitle
\begin{abstract}
We present new explicit upper bounds for the smoothness of the 
distribution of the random diagonal sum
$S_n=\sum_{j=1}^nX_{j,\pi(j)}$ 
of a random $n\times n$ matrix $X=(X_{j,r})$, where the $X_{j,r}$
are  independent integer valued random variables, and
$\pi$ denotes a uniformly distributed random permutation on 
$\{1,\dots,n\}$ independent of $X$. 
As a measure of smoothness, we consider the total variation distance 
between the distributions of $S_n$ and $1+S_n$. 
Our approach uses a new auxiliary inequality for a generalized 
normalized matrix hafnian, which could be of independent interest. 
This approach is also used to prove upper bounds of the L\'{e}vy 
concentration function of $S_n$ in the case of independent real valued random 
variables $X_{j,r}$. 
 
\end{abstract}
\section{Introduction and main result} \label{s257568}
Smoothness estimates and L\'{e}vy concentration function bounds
(the latter sometimes also called anti-concentration bounds) for probability 
distributions are often useful in the proofs of distributional 
approximations or limit theorems,
e.g.\ see 
\citet{MR0331448},
\citet{MR636780}, 
\citet{MR974089}, 
\citet{MR1368759}, 
\citet{MR1353441},
\citet[Proposition 4.6]{MR1716120}, 
\citet[Proposition 2.4]{MR3322321} and 
\citet[Lemma 4.1]{MR3647296}. 
 
As a measure of smoothness of the distribution $R$ of 
a random variable $Y$ with values in the set of integers $\ZZ$, we use 
the total variation distance between $Y$ and $1+Y$, namely
\begin{align*}
\dtv{Y}{1+Y}=\dtv{R}{\dirac_1*R}.
\end{align*}
Here $\dirac_x$ is the Dirac measure at point $x\in\RR$, where 
$\RR$ is the set of real numbers, $*$ denotes convolution, and   
\begin{align*}
\dtv{Y_1}{Y_2}
:=\dtv{R_1}{R_2}
:=\sup_{A\subseteq\ZZ}|R_1(A)-R_2(A)| 
\end{align*}
is the total variation distance between the
probability distributions $R_1$ and $R_2$ 
of the $\ZZ$-valued random variables $Y_1$ and $Y_2$, respectively. 
We note that $\dtv{R_1}{R_1}=\frac{1}{2}\|R_1-R_2\|$, where 
$\|\,\bfcdot\,\|$ denotes the total variation norm defined on the 
space of all finite signed measures on the power set of $\ZZ$. 

On the other hand, we consider the L\'{e}vy concentration function of 
the distribution $R$ of an $\RR$-valued random variable $Y$ given by
\begin{align*}
\conc{Y}{t}
:=\conc{R}{t}
:=\sup_{x\in\RR}R([x,x+t])
\quad \text{ for }t\in[0,\infty). 
\end{align*}

We need some further notation. 
For two sets $A$ and $B$, let $B^A=\{f\,|\,f:\,A\longrightarrow B\}$ 
be the set of all maps from $A$ to $B$ and 
$B_{\neq}^A=\{f\in B^A\,|\,f\text{ is injective}\}$.  
For $k,\ell\in\NN=\{1,2,3,\dots\}$, we write $\set{k}=\{1,\dots,k\}$,
$B^k=B^{\set{k}}$, $B^{k\times \ell}=B^{\set{k}\times \set{\ell}}$
and $B_{\neq}^k=B_{\neq}^{\set{k}}$. 
We write $f_j=f(j)$ for $f\in B^A$ and $j\in A$ 
and $f=(f_1,\dots, f_k)$ in the case $A=\set{k}$ for $k\in\NN$. 
In particular, for $k,\ell\in\NN$ with $k\leq \ell$, we have 
\begin{align*}
\set{\ell}_{\neq}^k
&=\{(j_1,\dots,j_k)\,|\,j_1,\dots,j_k\in\set{\ell} 
\mbox{ pairwise distinct}\}
\end{align*}
and $\set{k}_{\neq}^k$ is the set of all permutations on $\set{k}$.

Let $n\in\NN$ with $n\geq 2$ and $m=\floor{\frac{n}{2}}\in\NN$. 
Here, for $x\in\RR$, $\floor{x}$ is the largest integer $\leq x$. 
Let 
$X=(X_{j,r})$ be a random matrix in $\RR^{n\times n}$
such that the entries $X_{j,r}$ are independent random variables.
The distribution $P^{X_{j,r}}$ of $X_{j,r}$ is also denoted by 
$Q_{j,r}$.
Let $\pi$ be a random permutation on $\set{n}$ 
independent of $X$ and uniformly 
distributed on $\set{n}_{\neq}^n$. 
We call $(X_{1,\pi(1)},X_{2,\pi(2)},\dots,X_{n,\pi(n)})$ 
a (generalized) random diagonal of $X$ and 
\begin{align*}
S_n=\sum_{j=1}^nX_{j,\pi(j)}
\end{align*}
the corresponding random diagonal sum. In the recent literature, 
this is sometimes also called Hoeffding permutation statistic
(e.g.\ \citet{MR3920366}) 
or Hoeffding statistic (e.g.\ \citet{MR4372142}) in the case that  
the $X_{j,r}$ are constants. 

We note that the distribution 
\begin{align*}
P^{S_n}
&=\frac{1}{n!}\sum_{r\in\set{n}_{\neq}^n}\bigast_{j=1}^nQ_{j,r(j)}
\end{align*}
of $S_n$ remains the same 
and that the statements in Theorems \ref{t287432}--\ref{t98374587} 
below do not change if we weaken the independence assumption on the 
$X_{j,r}$ in such a way that, for each permutation
$r\in\set{n}_{\neq}^n$, the random variables in the  
diagonal $(X_{1,r(1)},X_{2,r(2)},\dots,X_{n,r(n)})$ are independent. 
For instance, this holds if the familiy of all columns 
(or all rows) of $X$ is independent. 

Sums like $S_n$ have important applications in the theory of rank tests 
(e.g.\ \citet{MR1680991}) and discrete probability 
(e.g.\ \citet[Chapter 4]{MR1163825}). 
Central limit theorems for $S_n$ under various conditions
in the case of constants $X_{j,r}$ have been 
obtained by \citet{MR11424}, \citet{MR44058}, \citet{MR89560}, 
and others. Berry-Esseen theorems can also be found in various papers,
see \citet{MR751577}, \citet{MR3322321}, \citet{MR4466417} 
and the references therein. 
The Poisson approximation in the case of Bernoulli random variables 
or integer valued constants $X_{j,r}$ has been investigated in 
\citet{MR385977}, \citet[Theorem 4.A and Remark 4.1.3]{MR1163825} 
and \citet{MR3920366}. 

In what follows, let  
\begin{align*}
R(j,k,r,s)
&=\frac{1}{2}(Q_{j,r}*Q_{k,s}+Q_{k,r}*Q_{j,s})
\quad\text{ for } (j,k),(r,s)\in\set{n}_{\neq}^2. 
\end{align*}
In other words, $R(j,k,r,s)$ is the distribution of 
$(1-I) (X_{j,r}+X_{k,s})+I(X_{k,r}+X_{j,s})$, where 
$I$ is a Bernoulli random variable with success probability 
$1/2$ and independent of
$X_{j,r}$, $X_{k,s}$, $X_{k,r}$, and $X_{j,s}$. 
For $\alpha\in(0,\infty)$, let 
\begin{align*} 
C(\alpha)=\sup_{x\in[0,\infty)}
\frac{\sqrt{\alpha+x}}{1+x}
\Big(\frac{x}{\sqrt{\alpha+1+x}}+\frac{1}{\sqrt{\alpha}}\Big). 
\end{align*}
\begin{remark}\label{r297582}
We note that $C(\alpha)$ coincides with 
$\widetilde{C}(\alpha,\frac{1}{2})$ in Corollary \ref{c9724722} below. 
As can easily be shown, it can be calculated numerically by using the 
fact that a corresponding maximizer $x$ satisfies
\begin{align*}
\lefteqn{(1+a)^2(1-3a)
+(1+a)(1-9a-6a^2)x}\\
&\quad{}
-(2+15a+15a^2+3a^3)x^2
-(2+10a+5a^2)x^3
+(1-2a)x^4
+x^5
=0.
\end{align*}
For instance, $C(\frac{1}{8})=1.9241\dots$ 
and $C(\frac{1}{4})=1.5593\dots$, where the 
maximizers are equal to 
$1.97329\dots$ and $2.544854\dots$, respectively. 
Remark \ref{r3286375} below implies that
$C(\alpha)$ is decreasing in $\alpha\in(0,\infty)$ and satisfies
$1\leq C(\alpha)\leq 1+\frac{1}{2\alpha}$. 
\end{remark}

The following three theorems form the main results of this paper. 
The proofs are given in Section \ref{s2864583}.
We recall that $m=\floor{\frac{n}{2}}$. 
\begin{theorem}\label{t287432}
Let the $X_{j,r}$ for $j,r\in\set{n}$ 
be $\ZZ$-valued, $\varepsilon\in(0,1]$, and
\begin{gather}
\nu(j,k,r,s)
:=1-\dtv{R(j,k,r,s)}{\dirac_1*R(j,k,r,s)}\leq \varepsilon
\quad \text{ for all } (j,k),(r,s)\in\set{n}_{\neq}^2, \label{e826628}\\
\xi
:=\frac{1}{n^2(n-1)^2}\sum_{(j,k)\in\set{n}_{\neq}^2}
\sum_{(r,s)\in\set{n}_{\neq}^2}\nu(j,k,r,s).\nonumber
\end{gather}
Then 
\begin{align}
\dtv{S_n}{1+S_n}
&\leq \frac{C(\frac{1}{8\varepsilon})}{\sqrt{\pi}}
\frac{1}{\sqrt{\frac{1}{8}+\floor{\frac{m}{2}}\xi}}. 
\label{e224527}
\end{align}
\end{theorem}

\begin{theorem}\label{t9117371}
Let the  $X_{j,r}$ for $j,r\in\set{n}$ be $\RR$-valued, 
$\varepsilon\in(0,1]$, $t\in[0,\infty)$, and
\begin{gather}
\zeta(j,k,r,s;t)
:=1-\conc{R(j,k,r,s)}{t}\leq \varepsilon
\quad \text{ for all }
(j,k),(r,s)\in\set{n}_{\neq}^2, \label{e297563}\\
\eta(t)
:=\frac{1}{n^2(n-1)^2}\sum_{(j,k)\in\set{n}_{\neq}^2}
\sum_{(r,s)\in\set{n}_{\neq}^2}\zeta(j,k,r,s;t).\nonumber
\end{gather}
Then 
\begin{align}
\conc{S_n}{t}
&\leq \frac{C(\frac{1}{8\varepsilon})}{\sqrt{\pi}}
\frac{1}{\sqrt{\frac{1}{8}+\floor{\frac{m}{2}}\eta(t)}}.\label{e2245221} 
\end{align}
\end{theorem}
\begin{remark}
Let the assumptions of Theorem \ref{t287432} or 
Theorem \ref{t9117371} hold. 
\begin{enumerate}\itemsep0pt

\item 
 Since 
$\nu(j,k,r,s),\zeta(j,k,r,s;t)\in[0,1]$ for all 
$(j,k),(r,s)\in\set{n}_{\neq}^2$, 
the inequalities in \eqref{e826628}  and  \eqref{e297563}
are valid for $\varepsilon=1$. In this case,
$\frac{C(\frac{1}{8\varepsilon})}{\sqrt{\pi}}
=\frac{C(\frac{1}{8})}{\sqrt{\pi}}
\leq 1.09$, see Remark \ref{r297582}. 
The upper bounds in \eqref{e224527} and \eqref{e2245221}
are of order $n^{-1/2}$ if $\xi$ and $\eta(t)$ are bounded away from 
zero. 

\item In \citet[Proposition 2.4]{MR3322321} another 
upper bound for $\conc{S_n}{t}$ can be found,
which is difficult to compare with \eqref{e2245221}. However, 
in contrast to \eqref{e2245221}, it cannot be useful for 
$t\geq 1$, since it contains a summand $c_1(n) t$, where $c_1(n)\geq1$ 
depends on $n$. 

\item By the properties of total variation distance and concentration 
function, we obtain that 
\begin{align*}
\nu(j,k,r,s)
&\geq 1-\frac{1}{2}(\dtv{Q_{j,r}}{\delta_1*Q_{j,r}}
+\dtv{Q_{j,s}}{\delta_1*Q_{j,s}}),\\
\zeta(j,k,r,s;t)
&\geq1-\frac{1}{2}(\conc{Q_{j,r}}{t}+\conc{Q_{j,s}}{t}),
\end{align*}
for $(j,k),(r,s)\in\set{n}_{\neq}^2$. 
Therefore, the right-hand sides of \eqref{e224527}
and \eqref{e2245221} can be further estimated. 
In fact,  $\xi$ and $\eta(t)$ used there
can be replaced with 
\begin{align*}
\frac{1}{n^2}
\sum_{(j,r)\in\set{n}^2} 
(1-\dtv{Q_{j,r}}{\dirac_1*Q_{j,r}})
\quad \text{ and }\quad 
\frac{1}{n^2}
\sum_{(j,r)\in\set{n}^2}(1-\conc{Q_{j,r}}{t}),
\end{align*}
respectively. 
However, in contrast to \eqref{e224527} and \eqref{e2245221}, 
the resulting inequalities are useless if the $X_{j,r}$ are constants. 
On the other hand, if the random matrix $X$ consists of identical 
columns, then 
$S_n=\sum_{j=1}^nX_{j,1}$ is a sum of independent random variables and 
the mentioned inequalities are comparable to those in 
Lemma \ref{l32521} below.  
\end{enumerate}
\end{remark}

In what follows, we assume that the $X_{j,r}$ are Bernoulli random 
variables. Let $\Be(p)=(1-p)\dirac_0+p\dirac_1$ be the Bernoulli 
distribution with success probability $p\in[0,1]$. 

\begin{theorem}\label{t98374587}
Let $p_{j,r}\in[0,1]$ and 
$Q_{j,r}=\Be(p_{j,r})$ for all $j,r\in\set{n}$.
For $(j,k),(r,s)\in\set{n}_{\neq}^2$, let
\begin{gather*}
u(j,k,r,s)
=p_{j,r}+p_{k,r}+p_{j,s}+p_{k,s},\quad
v(j,k,r,s)
=p_{j,r}p_{k,s}+p_{k,r}p_{j,s},\\
w(j,k,r,s)
=\frac{1}{2}u(j,k,r,s)-v(j,k,r,s),\\
\tau(j,k,r,s)
=\min\{w(j,k,r,s),1-w(j,k,r,s)\},\\
\sigma^2(j,k,r,s)
=\frac{u(j,k,r,s)}{2}+v(j,k,r,s)-\frac{u(j,k,r,s)^2}{4}
\end{gather*}
and let 
\begin{align*}
\chi
&=\frac{1}{n^2(n-1)^2}\sum_{(j,k)\in\set{n}_{\neq}^2}
\sum_{(r,s)\in\set{n}_{\neq}^2}
\tau(j,k,r,s),\quad
\psi
=\frac{1}{2n^2(n-1)^2}\sum_{(j,k)\in\set{n}_{\neq}^2}
\sum_{(r,s)\in\set{n}_{\neq}^2}\sigma^2(j,k,r,s).
\end{align*}
If $\tau(j,k,r,s)\leq\varepsilon$ for all 
$(j,k),(r,s)\in\set{n}_{\neq}^2$, then  
\begin{align}
\dtv{S_n}{1+S_n}
&\leq  \frac{C(\frac{1}{8\varepsilon})}{\sqrt{\pi}}
\frac{1}{\sqrt{\frac{1}{8}
+\floor{\frac{m}{2}}\chi}}.\label{e184376651}
\end{align}
If $\frac{1}{2}\sigma^2(j,k,r,s)\leq\varepsilon$ for all 
$(j,k),(r,s)\in\set{n}_{\neq}^2$, then  
\begin{align}
\conc{S_n}{0}
&\leq   \frac{C(\frac{1}{8\varepsilon})}{\sqrt{\pi}}
\frac{1}{\sqrt{\frac{1}{8}
+\floor{\frac{m}{2}}\psi}}.\label{e18437665}
\end{align}
\end{theorem}
\begin{remark} 
Let the assumptions of Theorem \ref{t98374587} hold. 
\begin{enumerate}

\item In contrast to \eqref{e224527} and \eqref{e2245221}, the 
upper bounds in \eqref{e184376651} and \eqref{e18437665} 
are explicit and do not contain the 
total variation distance or concentration function. The expression 
$\sigma^2(j,k,r,s)$ is the variance of $R(j,k,r,s)$ 
and we have $w(j,k,r,s)=R(j,k,r,s)(\{1\})$ 
for $(j,k),(r,s)\in\set{n}_{\neq}^2$, see Lemma \ref{l2320843}. 

\item Lemma \ref{l2320843} below implies that $\tau(j,k,r,s)$ and 
$\frac{1}{2}\sigma^2(j,k,r,s)$ are in $[0,\frac{1}{2}]$ for 
$(j,k),(r,s)\in\set{n}_{\neq}^2$. Hence, the inequalities in \eqref{e184376651} and \eqref{e18437665} are valid for 
$\varepsilon=\frac{1}{2}$. In this case, we get
$\frac{C(\frac{1}{8\varepsilon})}{\sqrt{\pi}}
=\frac{C(\frac{1}{4})}{\sqrt{\pi}}\leq 0.88$, see Remark \ref{r297582}.

\item The upper bound in \eqref{e184376651}
(resp.\ in \eqref{e18437665}) is of the same order as 
the one in \eqref{e224527} (resp.\ in \eqref{e2245221} for $t=0$), 
see Lemma \ref{l2320843}. 

\end{enumerate}
\end{remark}


\section{Proofs} \label{s2864583}
We need some preparation. 
Let the assumptions of Section \ref{s257568} hold 
if not stated otherwise. 
As usual, empty convolutions products and $0$-fold convolution powers 
of finite signed measures on the Borel $\sigma$-algebra over $\RR$ are 
defined to be $\dirac_0$. 
\begin{lemma}\label{l298735765}
For arbitrary $\varphi\in\set{n}_{\neq}^n$, we have 
\begin{align}
P^{S_n}
&=\frac{1}{n!}\sum_{r\in\set{n}_{\neq}^n}
\Big(\bigast_{\ell=1}^mR(\varphi(2\ell-1),
\varphi(2\ell),r(2\ell-1),r(2\ell))\Big)*Q_{\varphi(n),r(n)}^{*(n-2m)}.
\label{e23465}
\end{align}
\end{lemma}
\begin{proof}
This is easily shown using the identity theorem for 
characteristic functions and a general Laplace expansion 
for permanents of matrices with complex entries. 
For the latter see, for example,
\citet[page 559]{MR1440179} or \citet[Corollary 3.2]{Roos2020}. 
Alternatively, the following identity for $m'\in\set{m}$ 
is more general than \eqref{e23465} and can easily be shown 
inductively:
\begin{align*}
P^{S_n}
&=\frac{1}{n!}\sum_{r\in\set{n}_{\neq}^n}
\Big(\bigast_{\ell=1}^{m'}R(\varphi(2\ell-1),
\varphi(2\ell),r(2\ell-1),r(2\ell))\Big)*
\Big(\bigast_{j=2m'+1}^nQ_{\varphi(j),r(j)}\Big). 
\end{align*}
For $m'=m$, this yields \eqref{e23465}. 
\end{proof}

\begin{lemma}\label{l32521}
Let $Y_1,\dots,Y_n$ for  $n\in\NN$ be independent 
$\RR$-valued random variables and $T_n=\sum_{j=1}^nY_j$. 
Then, for $t\in[0,\infty)$, 
\begin{align}
\conc{T_n}{t}
\leq \sqrt{\frac{2}{\pii}}
\frac{1}{\sqrt{\frac{1}{4}+\sum_{j=1}^n(1-\conc{Y_j}{t})}}.
\label{e79825767}
\end{align}
If the $Y_1,\dots,Y_n$ are all  $\ZZ$-valued, then
\begin{align}
\dtv{T_n}{1+T_n}
\leq\sqrt{\frac{2}{\pii}}
\frac{1}{\sqrt{\frac{1}{4}+\sum_{j=1}^n(1-\dtv{Y_j}{1+Y_j})}}.
\label{e79825766}
\end{align}
\end{lemma}

\begin{proof}
See \citet[Corollaries 1.5 and 1.6]{MR2322695}. 
\end{proof}
It should be mentioned that  
\eqref{e79825767} is a special version of the Kolmogorov-Rogozin
inequality; e.g., see \citet[Theorem 2.1 on page 45]{MR974089}. 
Further, \eqref{e79825766} contains an 
improvement of an inequality of \citet[Proposition 4.6]{MR1716120}. 

Below we need an inequality for a generalized normalized hafnian. 
For $k,n\in\NN$ with $n\geq 2$ and $k\leq\floor{\frac{n}{2}}$ and
sets $A,B$ with cardinalities $\card{A}=k$ and $\card{B}=n$, 
we call the quantity 
\begin{align*}
\gnhaf(Z)
=\frac{(n-2k)!}{n!}\sum_{r\in B_{\neq}^{A\times\{1,2\}}}
\prod_{\ell\in A}z_{\ell,r(\ell,1),r(\ell,2)}
\end{align*}
the generalized normalized hafnian of the $3$-dimensional 
matrix $Z=(z_{\ell,r,s})\in\CC^{A\times B\times B}$. 
It is normalized in the sense that $\gnhaf(Z)=1$ 
if all the entries of $Z$ are equal to $1$.
\begin{remark}
If $k,n\in\NN$ with $n\geq 2$, $k\leq\floor{\frac{n}{2}}$ 
and $Z\in\CC^{\set{k}\times \set{n}\times \set{n}}$, then we can 
write  
\begin{align}\label{e276274}
\gnhaf(Z)
=\frac{(n-2k)!}{n!}\sum_{r\in \set{n}_{\neq}^{2k}}
\prod_{\ell=1}^kz_{\ell,r(2\ell-1),r(2\ell)}
=\frac{1}{n!}\sum_{r\in \set{n}_{\neq}^{n}}
\prod_{\ell=1}^kz_{\ell,r(2\ell-1),r(2\ell)}. 
\end{align}
If additionally $n=2k$ and $z_{\ell,r,s}=z_{1,r,s}=z_{1,s,r}$ 
for all $(\ell,r,s)\in\set{k}\times \set{n}\times\set{n}$, then 
\begin{align*}
\gnhaf(Z)
=\frac{1}{n!}\sum_{r\in \set{n}_{\neq}^{n}}
\prod_{\ell=1}^kz_{1,r(2\ell-1),r(2\ell)}
=\frac{k!2^k}{n!}\haf(Z'),
\end{align*}
where $\haf(Z')$ is the hafnian of the 2-dimensional matrix 
$Z'=(z'_{r,s})\in\CC^{n\times n}$ with $z'_{r,s}=z_{1,r,s}$
for all $r,s\in\set{n}$. For some properties of hafnians, see 
\citet{MR0464973} and \citet{MR3558532} and the references therein.
\end{remark}

\begin{lemma}\label{l3836}
Let $k,n\in\NN$ with $n\geq 2$ and  $k\leq\floor{\frac{n}{2}}$, 
$Z=(z_{\ell,r,s})\in\CC^{\set{k}\times \set{n}\times \set{n}}$. 
Then 
\begin{align}
|\gnhaf(Z)|
&\leq \prod_{\ell=1}^k\Big(
\frac{1}{n(n-1)}\sum_{(r,s)\in\set{n}_{\neq}^2}
\Big|\frac{1}{2}(z_{\ell,r,s}+z_{\ell,s,r})\Big|^2\Big)^{1/2}
\label{e82776586}\\
&\leq \prod_{\ell=1}^k
\Big(\frac{1}{n(n-1)}\sum_{(r,s)\in\set{n}_{\neq}^2}
|z_{\ell,r,s}|^2\Big)^{1/2}.\label{e82776587}
\end{align}
\end{lemma}
\begin{remark} 
\begin{enumerate}

\item Lemma \ref{l3836} generalizes an inequality for hafnians, 
see \citet[formula (42)]{Roos2020}. 

\item In \eqref{e82776586} and \eqref{e82776587}, 
equality holds if for every $\ell\in\set{k}$ 
all the $z_{\ell,r,s}$ for $(r,s)\in\set{n}_{\neq}^2$ are 
identical. 

\end{enumerate}

\end{remark}

For the proof of Lemma \ref{l3836}, we need further notation and 
another lemma. 
If $A'\subseteq A$, $B'\subseteq B$ and $C'\subseteq C$  are all 
non-empty sets and $Z\in\CC^{A\times B\times C}$,  then 
$Z[A',B',C']$ denotes the submatrix of $Z$ with 
entries $z_{\ell,r,s}$ for $\ell\in A'$, $r\in B'$ and $s\in C'$. 
For a set $A$ and $k\in\Zpl=\{0,1,2,\dots\}$, let 
\begin{align*}
\Subs{A}{k}=\{A'\,|\,A'\subseteq A,\,\card{A'}=k\}
\end{align*}
be the set of all subsets of $A$ with cardinality $\card{A'}=k$. 
For $d\in\NN$ and $n\in\Zpl$, let
\begin{align*}
\Comp(n,d)
=\Big\{w=(w_1,\dots,w_d)\in\Zpl^d\,\Big|\,\sum_{\ell=1}^dw_\ell=n\Big\}
\end{align*}
be the set of the so-called weak $d$-compositions of $n$. 
For a set $A$ with $\card{A}=n$ and 
$w\in\Comp(n,d)$, let 
\begin{align*}
\Part(A,w)
=\Big\{W=(W_1,\dots,W_d) \,\Big|\, 
W_\ell\in\Subs{A}{w_\ell},\,(\ell\in\set{d})
\mbox{ pairwise disjoint}\Big\}.
\end{align*}
It is clear that  $\bigcup_{\ell=1}^dW_\ell=A$ for $W\in\Part(A,w)$. 
Every $W\in\Part(A,w)$ is called an
ordered weak partition of $A$ of type $w$. 

\begin{lemma}\label{l29757659}
Let $d,n\in\NN$, 
$w=(w_{1},\dots,w_{d})\in\Comp(n,d)$,  
$g_\ell:\,\Subs{\set{n}}{w_\ell}\longrightarrow\CC$ 
be a map for all $\ell\in\set{d}$. Then 
\begin{align*}
\ABS{\sum_{W\in\Part(\set{n},w)}
\prod_{\ell=1}^dg_\ell(W_\ell)} 
&\leq \frac{n!}{w!}\prod_{\ell=1}^d\Big(\frac{1}{\binomial{n}{w_\ell}}
\sum_{W_\ell\in\Subs{\set{n}}{w_\ell}}\abs{g_\ell(W_\ell)}^2\Big)^{1/2}.
\end{align*}
\end{lemma}
\begin{proof}
This follows from the more general Corollary 3.1 in 
\citet{Roos2020}.
\end{proof}

\begin{proof}[Proof of Lemma \ref{l3836}]
Let $d=n-k=k+n-2k$ and $w=(w_1,\dots,w_k,w_{k+1},\dots,w_d)\in\Zpl^d$
with $w_\ell=2$ if $\ell\in\set{k}$ and $w_\ell=1$ if 
$\ell\in\set{d}\setminus\set{k}$. Since $\sum_{\ell=1}^dw_\ell=n$,
we have $w\in\Comp(n,d)$.
For $\ell\in\set{k}$, $(r,s)\in\set{n}_{\neq}^2$ and $W=\{r,s\}$, let
\begin{align*}
g_\ell(W)
&=\gnhaf(Z[\{\ell\},W,W])
=\frac{1}{2}(z_{\ell,r,s}+z_{\ell,s,r}).
\end{align*}
For $\ell\in\set{d}\setminus\set{k}$, $r\in\set{n}$ and $W=\{r\}$, 
let $g_\ell(W)=1$. Using \eqref{e276274}, it is easily shown 
inductively that, for $k'\in\set{k}$,  
\begin{align*}
\gnhaf(Z)
&=\frac{(n-2k)!}{n!}\sum_{r\in\set{n}_{\neq}^{2k}}
\Big(\prod_{\ell=1}^{k'}\Big( \frac{1}{2}(z_{\ell,r(2\ell-1),r(2\ell)}
+z_{\ell,r(2\ell),r(2\ell-1)})\Big)\Big)
\prod_{\ell=k'+1}^{k}z_{\ell,r(2\ell-1),r(2\ell)}.
\end{align*}
Letting $k'=k$, this yields
\begin{align*}
\gnhaf(Z)
&=\frac{(n-2k)!}{n!}\sum_{r\in\set{n}_{\neq}^{2k}}
\prod_{\ell=1}^kg_\ell(\{r(2\ell-1),r(2\ell)\})\\
&=\frac{1}{n!}\sum_{r\in\set{n}_{\neq}^{n}}
\prod_{\ell=1}^kg_\ell(\{r(2\ell-1),r(2\ell)\})\\
&=\frac{2^k}{n!}\sum_{W\in\Part(\set{n},w)}
\prod_{\ell=1}^dg_\ell(W_\ell).
\end{align*}
From Lemma \ref{l29757659}, we obtain that 
\begin{align*}
|\gnhaf(Z)|
&\leq  \prod_{\ell=1}^d\Big(\frac{1}{\binomial{n}{w_\ell}}
\sum_{W_\ell\in\Subs{\set{n}}{w_\ell}}|g_\ell(W_\ell)|^2\Big)^{1/2}
=\prod_{\ell=1}^k\Big(\frac{1}{\binomial{n}{2}}
\sum_{W_\ell\in\Subs{\set{n}}{2}}|g_\ell(W_\ell)|^2\Big)^{1/2}\\
&=\prod_{\ell=1}^k\Big(\frac{1}{2\binomial{n}{2}}
\sum_{(r,s)\in\set{n}_{\neq}^2}|g_\ell(\{r,s\})|^2\Big)^{1/2}\\
&=\prod_{\ell=1}^k\Big(\frac{1}{n(n-1)}
\sum_{(r,s)\in\set{n}_{\neq}^2}
\Big|\frac{1}{2}(z_{\ell,r,s}+z_{\ell,s,r})\Big|^2\Big)^{1/2},
\end{align*}
which shows \eqref{e82776586}. 
Inequality \eqref{e82776587} is clear. 
\end{proof}

\begin{lemma}\label{l296978}
Let $\alpha,\beta\in(0,\infty)$ 
and $Y$ be a non-negative random variable with finite expectation
$\mu:=\EE Y$ and variance $\sigma^2:=\Var Y$. 
Then, for all $c\in(0,\infty)$,
\begin{align}\label{e3649322}
\EE \frac{1}{(\alpha+Y)^\beta}
&\leq \frac{1}{(\alpha+c)^\beta}
+\frac{\beta(c-\mu)}{(\alpha+c)^{\beta+1}}
+\frac{\sigma^2+(c-\mu)^2}{c^2}
\Big(\frac{1}{\alpha^\beta}-\frac{\alpha+c(\beta+1)}{(\alpha+c)^{
\beta+1}}\Big).
\end{align}

\end{lemma}
\begin{proof}
For $x\in(-\alpha,\infty)$, let 
$f(x)=\frac{1}{(\alpha+x)^\beta}$. The first two derivatives of $f$ 
are given by $f'(x)=\frac{-\beta}{(\alpha+x)^{\beta+1}}$ and
$f''(x)=\frac{\beta(\beta+1)}{(\alpha+x)^{\beta+2}}$. Therefore, 
$f''$ is decreasing and, using Taylor's formula, we obtain 
\begin{align*}
\EE f(Y)
&=f(c)+(\mu-c)f'(c)
+\EE\Big((Y-c)^2\int_0^1(1-t)f''((1-t)c+tY)\,\dd t\Big)\\
&\leq f(c)+(\mu-c)f'(c)+\EE(Y-c)^2\int_0^1(1-t)f''((1-t)c)\,\dd t\\
&= \frac{1}{(\alpha+c)^\beta}
+\frac{\beta(c-\mu)}{(\alpha+c)^{\beta+1}}
+(\sigma^2+(c-\mu)^2)
\int_0^1\frac{t\beta(\beta+1)}{(\alpha+tc)^{\beta+2}}\,\dd t,
\end{align*}
where 
\begin{align*}
\int_0^1\frac{t\beta(1+\beta)}{(\alpha+tc)^{\beta+2}}\,\dd t
&=\frac{1}{c^2}\Big(\frac{1}{\alpha^\beta}
-\frac{\alpha+c(\beta+1)}{(\alpha+c)^{\beta+1}}\Big).
\qedhere
\end{align*}
\end{proof}

\begin{remark} 
Let the assumptions of Lemma \ref{l296978} hold. 
The optimal constant $c$ in \eqref{e3649322}
is $c=\mu+\frac{\sigma^2}{\mu}$, giving
\begin{align}\label{e8616475}
\EE \frac{1}{(\alpha+Y)^\beta}
&\leq \frac{\kappa^\beta\mu^2+\sigma^2/\alpha^\beta}{\mu^2+\sigma^2}, 
\quad
\mbox{ where }\kappa=\frac{\mu}{\sigma^2+\mu(\alpha+\mu)}.
\end{align}
This was shown by \citet{MR816102} using a recursion 
and the Jensen inequality in the case $\beta\in\NN$. 
Another proof was given by \citet{MR1072493} avoiding a recursion. 
We prefer \eqref{e3649322} over \eqref{e8616475}, 
since \eqref{e3649322} is valid for all positive $c$, which 
simplifies the proof of the next corollary. 
\end{remark}

\begin{corollary}\label{c9724722}
Let $\alpha\in(0,\infty)$, $\beta\in(0,1]$ and
$Y$ be a non-negative random variable with 
$\EE Y^2<\infty$.  
If $\Var Y \leq \EE Y=:\mu$, then 
\begin{align}\label{e97271}
\EE \frac{1}{(\alpha+Y)^\beta}
&\leq \frac{\widetilde{C}(\alpha,\beta)}{(\alpha+\mu)^\beta},
\end{align}
where 
$\widetilde{C}(\alpha,\beta)
=\sup_{x\in[0,\infty)} h(x)$ and 
$h(x):=h_{\alpha,\beta}(x)
:=\frac{(\alpha+x)^\beta}{1+x}(\frac{x}{(\alpha+1+x)^\beta}
+\frac{1}{\alpha^\beta})$.
For $\alpha$ and $\beta$ being fixed, the factor 
$\widetilde{C}(\alpha,\beta)$ in \eqref{e97271} 
cannot be made smaller. 

\end{corollary}

\begin{proof}
From  \eqref{e3649322} with $c=1+\mu$, we obtain that 
\begin{align}
\EE \frac{1}{(\alpha+Y)^\beta}
&\leq \frac{1}{(\alpha+\mu)^\beta}h(\mu), \label{e926493}
\end{align}
which implies \eqref{e97271}.

To prove the optimality of the factor 
$\widetilde{C}(\alpha,\beta)$, we consider $Y$ satisfying
$P(Y=d)=p=1-P(Y=0)$, 
where $d\in(0,\infty)$ and $p\in[0,1]$. 
Then $\mu=pd$ and $\Var X
=\mu(d-\mu)$.  
We assume that $\mu=\Var Y$, that is, $d=1+\mu$, 
$p=\frac{\mu}{1+\mu}$, $1-p=\frac{1}{1+\mu}$. 
Consequently, in \eqref{e926493}, equality holds. 
Since $\mu\in(0,\infty)$ can be chosen arbitrarily,
$\widetilde{C}(\alpha,\beta)$ cannot be made smaller 
in \eqref{e97271}.
\end{proof}
\begin{remark}\label{r3286375}
Let $\alpha\in(0,\infty)$, $\beta\in(0,1]$
and $\widetilde{C}(\alpha,\beta)$ be defined as in 
Corollary \ref{c9724722}. Then 
$\widetilde{C}(\alpha,\beta)$ is decreasing in $\alpha$ 
and satisfies 
$1\leq \widetilde{C}(\alpha,\beta)\leq 1+\frac{\beta}{\alpha}$. 

\end{remark}
\begin{proof}
Let 
$h_{\alpha,\beta}(x)$
for $x\in[0,\infty)$, $\alpha\in(0,\infty)$ and $\beta\in(0,1]$ 
be as in Corollary \ref{c9724722}.
Then 
\begin{align*}
h_{\alpha,\beta}(x)
&=\frac{1}{1+x}\Big(x\Big(1-\frac{1}{\alpha+1+x}\Big)^\beta
+\Big(1+\frac{x}{\alpha}\Big)^\beta\Big),\\
\frac{\partial }{\partial\alpha}h_{\alpha,\beta}(x)
&=\frac{x\beta}{1+x}\Big(
\frac{1}{(1-\frac{1}{\alpha+1+x})^{1-\beta}(\alpha+1+x)^2}
-\frac{1}{(1+\frac{x}{\alpha})^{1-\beta}\alpha^2}\Big)\\
&= \frac{x\beta}{(1+x)(\alpha+x)^{1-\beta}}\Big(
\frac{1}{(\alpha+1+x)^{1+\beta}}-\frac{1}{\alpha^{1+\beta}}\Big)
\leq 0.
\end{align*}
For $\alpha_1,\alpha_2\in(0,\infty)$ with $\alpha_2\leq \alpha_1$, 
we obtain 
$h_{\alpha_1,\beta}(x)
\leq h_{\alpha_2,\beta}(x)
\leq 
\widetilde{C}(\alpha_2,\beta)$
and hence 
$\widetilde{C}(\alpha_1,\beta)
\leq\widetilde{C}(\alpha_2,\beta)$. 
On the other hand, we have $h_{\alpha,\beta}(0)=1$ and by using 
that $(1+y)^\beta\leq 1+\beta y$ for $y\in[-1,\infty)$, we obtain
\begin{align*}
h_{\alpha,\beta}(x)
&\leq\frac{1}{1+x}\Big(x\Big(1-\frac{\beta}{\alpha+1+x}\Big)
+1+\frac{\beta x}{\alpha}\Big)\\
&=1+\frac{\beta x}{1+x}\Big(\frac{1}{\alpha}-\frac{1}{\alpha+1+x}\Big)
\leq 1+\frac{\beta x}{\alpha(1+x)}
\leq 1+\frac{\beta}{\alpha}.\qedhere
\end{align*}
\end{proof}

We now show that Theorem \ref{t287432} is a consequence of the 
following more general theorem.

\begin{theorem}\label{t9117372}
Let the $X_{j,r}$ for $j,r\in\set{n}$  
be $\ZZ$-valued and $\widetilde{\nu}(j,k,r,s)\in[0,1]$ satisfy
\begin{align*} 
\nu(j,k,r,s)
:=1-\dtv{R(j,k,r,s)}{\dirac_1*R(j,k,r,s)}
\geq\widetilde{\nu}(j,k,r,s)
\end{align*}
for all $(j,k),(r,s)\in\set{n}_{\neq}^2$.
Let $\varphi\in\set{n}_{\neq}^n$ be arbitrary and
\begin{align}\label{e82663}
\xi_q:=\xi_q(\varphi)
:=\frac{1}{mn(n-1)}\sum_{\ell\in\set{m}}
\sum_{(r,s)\in\set{n}_{\neq}^2}
(\widetilde{\nu}(\varphi(2\ell-1),\varphi(2\ell),r,s))^q
\quad\text{ for } q\in\{1,2\}. 
\end{align}
Let $\varepsilon\in(0,1]$ satisfy
$\xi_2-\xi_1^2\leq \varepsilon \xi_1$. Then
\begin{align}
\dtv{S_n}{1+S_n}
&\leq \frac{C(\frac{1}{8\varepsilon})}{\sqrt{\pi}}
\frac{1}{\sqrt{\frac{1}{8}+\floor{\frac{m}{2}}\xi_1}}. 
\label{e224523}
\end{align}
\end{theorem}

\begin{proof}[Proof of Theorem \ref{t287432}]
Let $\widetilde{\nu}(j,k,r,s):=\nu(j,k,r,s)$ 
for all $(j,k),(r,s)\in\set{n}_{\neq}^2$. 
From \eqref{e826628}, it follows that, using Theorem
\ref{t9117372} and its notation, \eqref{e224523} holds 
for all $\varphi\in\set{n}_{\neq}^n$. 
Therefore
\begin{align*}
\dtv{S_n}{1+S_n}
&\leq \frac{C(\frac{1}{8\varepsilon})}{\sqrt{\pi}}
\frac{1}{\sqrt{\frac{1}{8}+\floor{\frac{m}{2}}
\frac{1}{n!}\sum_{\varphi\in\set{n}_{\neq}^n}\xi_1(\varphi)}},
\end{align*}
where $\frac{1}{n!}\sum_{\varphi\in\set{n}_{\neq}^n}\xi_1(\varphi)
=\xi$.  
\end{proof}

\begin{proof}[Proof of Theorem \ref{t9117372}]
Using Lemmata \ref{l298735765} and \ref{l32521}, we obtain
for arbitrary $\varphi\in\set{n}_{\neq}^n$ that  
\begin{align*}
\lefteqn{\dtv{S_n}{S_n+1}}\\
&=\frac{1}{2}\big\|(\dirac_0-\dirac_1)*P^{S_n}\big\|\\
&=\frac{1}{2}\Big\|\frac{1}{n!}\sum_{r\in\set{n}_{\neq}^n}
(\dirac_0-\dirac_1)*
\Big(\bigast_{\ell=1}^mR(\varphi(2\ell-1),
\varphi(2\ell),r(2\ell-1),r(2\ell))
\Big)*Q_{\varphi(n),r(n)}^{*(n-2m)}\Big\| \\
&\leq \frac{1}{\,n!}\sum_{r\in\set{n}_{\neq}^n}
\frac{1}{2}\Big\|(\dirac_0-\dirac_1)*
\Big(\bigast_{\ell=1}^mR(\varphi(2\ell-1),
\varphi(2\ell),r(2\ell-1),r(2\ell))
\Big)\Big\| \\
&\leq \sqrt{\frac{2}{\pii}}
\frac{1}{n!}\sum_{r\in\set{n}_{\neq}^n}
\frac{1}{\sqrt{\frac{1}{4}+\sum_{\ell=1}^m
\nu(\varphi(2\ell-1),\varphi(2\ell),r(2\ell-1),r(2\ell))}}\\
&\leq\frac{1}{\sqrt{\pii}}
\EE\frac{1}{\sqrt{\frac{1}{8}+L(\varphi)}},
\end{align*}
where 
$L(\varphi)=\frac{1}{2}\sum_{\ell=1}^m
\widetilde{\nu}(\varphi(2\ell-1),
\varphi(2\ell),\pi(2\ell-1),\pi(2\ell))$. 
We now use the representation 
\begin{align}\label{e82947}
\frac{1}{\sqrt{b}}
=\frac{1}{\Gamma(1/2)}\int_0^\infty \frac{1}{\sqrt{x}}\ee^{-xb}\,\dd x
\end{align}
for $b\in(0,\infty)$, which implies that 
\begin{align*}
\lefteqn{
\dtv{S_n}{S_n+1}}\\
&=\frac{1}{\sqrt{\pii}}
\frac{1}{\Gamma(1/2)}\int_0^\infty 
\frac{\ee^{-x/8}}{\sqrt{x}}\EE\ee^{-xL(\varphi)}\,\dd x\\
&=\frac{1}{\sqrt{\pii}}
\frac{1}{\Gamma(1/2)}\int_0^\infty \frac{\ee^{-x/8}}{\sqrt{x}}
\frac{1}{n!}\sum_{r\in\set{n}_{\neq}^{n}}
\prod_{\ell=1}^m\exp\Bigl(-\frac{x}{2}
\widetilde{\nu}(\varphi(2\ell-1),
\varphi(2\ell),r(2\ell-1),r(2\ell))\Bigr)\,\dd x.
\end{align*}
Lemma \ref{l3836} implies that
\begin{align}
\lefteqn{\dtv{S_n}{S_n+1}}\nonumber\\
&\leq \frac{1}{\sqrt{\pii}}
\frac{1}{\Gamma(1/2)}\int_0^\infty \frac{\ee^{-x/8}}{\sqrt{x}}
\prod_{\ell=1}^m\Big(\frac{1}{n(n-1)}\sum_{(r,s)\in\set{n}_{\neq}^2}
\exp(-x\widetilde{\nu}(\varphi(2\ell-1),
\varphi(2\ell),r,s))\Big)^{1/2}\,\dd x.
\label{e2826486}
\end{align}
Using \eqref{e2826486},  
the inequality between arithmetic and geometric means and 
\eqref{e82947} again, 
we obtain
\begin{align*}
\lefteqn{\dtv{S_n}{S_n+1}}\\
&\leq \frac{1}{\sqrt{\pii}}
\frac{1}{\Gamma(1/2)}\int_0^\infty \frac{\ee^{-x/8}}{\sqrt{x}}
\Big(\frac{1}{mn(n-1)}\\
&\quad{}\times
\sum_{\ell\in\set{m}}
\sum_{(r,s)\in\set{n}_{\neq}^2}
\exp(-x\widetilde{\nu}(\varphi(2\ell-1),
\varphi(2\ell),r,s))\Big)^{\floor{m/2}}\,\dd x\\
&=\frac{1}{\sqrt{\pii}}
\frac{1}{\Gamma(1/2)}\int_0^\infty \frac{\ee^{-x/8}}{\sqrt{x}}
\EE\exp(-xM)\,\dd x\\
&=\frac{1}{\sqrt{\pii}}\EE\frac{1}{\sqrt{\frac{1}{8}+M}}, 
\end{align*}
where 
$M
:=M(\varphi)
:=\sum_{j=1}^{\floor{m/2}}M_j$ and the 
$M_j:=M_j(\varphi)$ for $j\in\{1,\dots,\floor{m/2}\}$ 
are independent random variables with distribution 
\begin{align*}
P^{M_j}
=\frac{1}{m n(n-1)}\sum_{\ell\in\set{m}}
\sum_{(r,s)\in\set{n}_{\neq}^2}
\dirac_{\widetilde{\nu}(\varphi(2\ell-1),\varphi(2\ell),r,s)}.
\end{align*}
We clearly have 
$\EE M=\floor{m/2}\xi_1$ and 
$\Var(\frac{1}{\varepsilon}M)
\leq \EE(\frac{1}{\varepsilon}M)$. 
Therefore, Corollary \ref{c9724722} yields
\begin{align*}
\dtv{S_n}{S_n+1}
&\leq \frac{1}{\sqrt{\varepsilon\pii}}
\EE\frac{1}{\sqrt{\frac{1}{8\varepsilon}+\frac{1}{\varepsilon}
M}}
\leq \frac{1}{\sqrt{\varepsilon\pii}}
\frac{\widetilde{C}(\frac{1}{8\varepsilon},\frac{1}{2})}{
\sqrt{\frac{1}{8\varepsilon}+\EE(\frac{1}{\varepsilon} M)}}
=\frac{C(\frac{1}{8\varepsilon})}{\sqrt{\pii}}
\frac{1}{\sqrt{\frac{1}{8}+\floor{\frac{m}{2}}\xi_1}}.\qedhere
\end{align*}
\end{proof}

Theorem \ref{t9117371} is an easy consequence of 
the following more general theorem. 
\begin{theorem}\label{t22628362}
Let the  $X_{j,r}$ for $j,r\in\set{n}$ be $\RR$-valued, 
$t\in[0,\infty)$ and
$\widetilde{\zeta}(j,k,r,s;t)\in[0,1]$ satisfy
\begin{align*}
\zeta(j,k,r,s;t)
:=1-\conc{R(j,k,r,s)}{t}\geq\widetilde{\zeta}(j,k,r,s;t)
\end{align*}
for all $(j,k),(r,s)\in\set{n}_{\neq}^2$. 
Let $\varphi\in\set{n}_{\neq}^n$ be arbitrary and
\begin{align*} 
\eta_q(t):=\eta_q(\varphi,t)
:=\frac{1}{mn(n-1)}\sum_{\ell\in\set{m}}
\sum_{(r,s)\in\set{n}_{\neq}^2}
(\widetilde{\zeta}(\varphi(2\ell-1),\varphi(2\ell),r,s;t))^q
\quad\text{ for }q\in\{1,2\}. 
\end{align*}
Let $\varepsilon\in(0,1]$ satisfy
$\eta_2(t)-\eta_1(t)^2\leq\varepsilon\eta_1(t)$. Then
\begin{align*}
\conc{S_n}{t}
&\leq \frac{C(\frac{1}{8\varepsilon})}{\sqrt{\pi}}
\frac{1}{\sqrt{\frac{1}{8}+\floor{\frac{m}{2}}\eta_1(t)}}.
\end{align*}
\end{theorem}
The proofs of Theorems \ref{t9117371} and \ref{t22628362} are 
analogous to those of  
Theorems \ref{t287432} and \ref{t9117372} and are therefore omitted. 
For the proof of Theorem \ref{t98374587}, we need the following lemma. 
\begin{lemma}\label{l2320843}
Let $a,b,c,d\in[0,1]$, 
$R=\frac{1}{2}(\Be(a)*\Be(b)+\Be(c)*\Be(d))$. 
Let 
$u=a+b+c+d$, $v=ab+cd$, and $w=\frac{u}{2}-v$.
Then $\sigma^2=\frac{u}{2}+v-\frac{u^2}{4}\in[0,1]$ 
is the variance of $R$, $w=R(\{1\})\in[0,1]$,
\begin{align}
\dtv{R}{\delta_1*R}
&=\frac{1}{2}(1-w)+\frac{1}{4}|2-4w-v|+\frac{1}{4}|2w-v|,
\label{e22579}\\
\conc{R}{0}&=\max\Big\{1-w-\frac{v}{2},w,\frac{v}{2}\Big\}
\label{e22580}
\end{align}
and
\begin{gather}
\min\{w,1-w\}
\leq1-\dtv{R}{\delta_1*R}
\leq\min\{2w,1-w\},\label{e87165386}\\
\frac{\sigma^2}{2}
\leq1-\conc{R}{0}
\leq 2\sigma^2.\label{e287211}
\end{gather}
\end{lemma}
\begin{proof}
The identities \eqref{e22579} and \eqref{e22580} 
follow directly from the observation that 
\begin{align*}
R
&=\Big(1-w-\frac{v}{2}\Big)\dirac_0+w\dirac_1+\frac{v}{2}\dirac_2.
\end{align*}
In what follows, we use the simple identities
\begin{align}
1-\dtv{R}{\delta_1*R}
&=\left\{
\begin{array}{ll}
w+\frac{v}{2}, &\mbox{ if } v\leq 2w\leq 1-\frac{v}{2},\\
1-\frac{v}{2}, &\mbox{ if } 1-\frac{v}{2}\leq 2w\leq v,\\
2w,  &\mbox{ if } 2w\leq \min\{1-\frac{v}{2},v\},\\
1-w, &\mbox{ if } 2w\geq \max\{1-\frac{v}{2},v\},
\end{array}
\right. \label{e1862518}\\
1-\conc{R}{0}
&=\min\Big\{\frac{u-v}{2},1-\frac{u}{2}+v,1-\frac{v}{2}\Big\},
\nonumber
\end{align}
and the inequalities
\begin{align} \label{e21865}
\max\{u-2,0\}
\leq v
\leq \min\Big\{\frac{u^2}{4}, 2-u+\frac{u^2}{4}\Big\}.
\end{align}
The latter ones can be shown as follows. 
On the one hand, we have 
\begin{align*}
u-v
=a(1-b)+b+c(1-d)+d
\leq (1-b)+b+(1-d)+d=2,
\end{align*}
which implies the first inequality in \eqref{e21865}. 
On the other hand,
\begin{align*}
u^2-4v
&=(a+b+c+d)^2-4ab-4cd\\
&\geq a^2-2ab+b^2+c^2-2cd+d^2
=(a-b)^2+(c-d)^2\geq0 
\end{align*}
giving
$v\leq \frac{u^2}{4}$. Using the latter inequality for $1-a,1-b,1-c,1-d$ 
instead of $a,b,c,d$, respectively, we obtain 
$2-u+v
=(1-a)(1-b)+(1-c)(1-d)
\leq\frac{1}{4}(4-u)^2
=\frac{1}{4}(16-8u+u^2)$, which yields
$v\leq 2-u+\frac{u^2}{4}$. This proves the second inequality 
in \eqref{e21865}. 

The expectation of $R$ is $\frac{u}{2}$ and the
variance of $R$ satisfies
\begin{align*}
\sigma^2
&=w+2v-\frac{u^2}{4}
=\frac{u}{2}+v-\frac{u^2}{4}\\
&\leq \frac{u}{2}
+\min\Big\{\frac{u^2}{4},2-u+\frac{u^2}{4}\Big\}-\frac{u^2}{4}
\leq\min\Big\{\frac{u}{2},2-\frac{u}{2}\Big\}
\leq1.
\end{align*}
Further, the inequalities in \eqref{e87165386} easily follow 
from \eqref{e1862518} and \eqref{e21865}. 

Let us now show the first inequality in \eqref{e287211}. 
If $u\leq 2$ then $\frac{u^2}{4}\leq \frac{u}{2}$.
If $u\geq2$ then $2-u+\frac{u^2}{4}\leq\frac{u}{2}$. 
Hence, from \eqref{e21865}, we get
$v\leq \min\{\frac{u}{2},\frac{u^2}{4}\}$. 
Therefore
$\frac{\sigma^2}{2}
=\frac{1}{2}(\frac{u}{2}+v-\frac{u^2}{4})
\leq \frac{1}{2}(u-\frac{u^2}{4})
\leq \frac{u-v}{2}$. 
From \eqref{e21865}, we obtain 
$-\frac{u^2}{2}+3u-4
=\frac{1}{2}(u-2)(4-u)\leq \max\{u-2,0\}\leq v$,
which implies that $\frac{u}{2}\leq 2-u+\frac{v}{2}+\frac{u^2}{4}$.
Therefore, 
$\frac{\sigma^2}{2}
=\frac{1}{2}(\frac{u}{2}+v-\frac{u^2}{4})
\leq \frac{1}{2}(
2-u+\frac{v}{2}+v)
\leq 1-\frac{u}{2}+v$.
For $u\leq 2$, we have 
$\frac{u^2}{4}
\leq1-\frac{u}{4}+\frac{u^2}{8}$. For $u\geq2$, we have 
$2-u+\frac{u^2}{4}
\leq1-\frac{u}{4}+\frac{u^2}{8}$. 
From \eqref{e21865}, we get that 
$v\leq 1-\frac{u}{4}+\frac{u^2}{8}$. 
Therefore 
\begin{align*}
\frac{\sigma^2}{2}
=\frac{1}{2}\Big(\frac{u}{2}+2v-v-\frac{u^2}{4}\Big)
\leq
\frac{1}{2}
\Big(\frac{u}{2}+2\Big(1-\frac{u}{4}+\frac{u^2}{8}\Big)
-v-\frac{u^2}{4}\Big)
=1-\frac{v}{2}. 
\end{align*}
This shows the first inequality in \eqref{e287211}. 
For the second one, we show that 
\begin{align*}
\min\Big\{\frac{u-v}{2},1-\frac{u}{2}+v,1-\frac{v}{2}\Big\}
\leq 2\sigma^2=u+2v-\frac{u^2}{2}
\end{align*}
or equivalently
\begin{align*}
\min\Big\{\frac{u}{2}-\frac{5}{2}v,1-\frac{u}{2}-v,1-\frac{5}{2}v\Big\}
\leq u\Big(1-\frac{u}{2}\Big).
\end{align*}
Using \eqref{e21865},  we obtain  
\begin{align}
\lefteqn{\min\Big\{\frac{u}{2}-\frac{5}{2}v,
1-\frac{u}{2}-v,1-\frac{5}{2}v\Big\}}\nonumber\\
&\leq \min\Big\{\frac{u}{2}-\frac{5}{2}\max\{u-2,0\},
1-\frac{u}{2}-\max\{u-2,0\},1-\frac{5}{2}\max\{u-2,0\}\Big\}.
\label{e927292}
\end{align}
If $u\leq 2$, then the right-hand side in \eqref{e927292} 
is equal to $\min\{\frac{u}{2},1-\frac{u}{2}\}
\leq u(1-\frac{u}{2})$.
If $u\geq 2$, then it is equal to
\begin{align*}
\min\Big\{\frac{u}{2}-\frac{5}{2}(u-2),
1-\frac{u}{2}-u+2,1-\frac{5}{2}(u-2)\Big\} 
&=\min\Big\{5-2u,3-\frac{3}{2}u,6-\frac{5}{2}u\Big\}\\
&=\min\Big\{3-\frac{3}{2}u,6-\frac{5}{2}u\Big\}
\leq u\Big(1-\frac{u}{2}\Big).
\end{align*}
This shows the second inequality in \eqref{e287211} and 
completes the proof of the lemma. 
\end{proof}
\begin{remark}
The second inequality in \eqref{e287211} is  better than 
the inequality $1-\conc{R}{0}\leq 4\sigma^2$, which is
a simple consequence of the Chebyshev inequality, 
see \citet[Proposition 1.6.1 on page 27]{MR0331448}.
\end{remark}

The following corollary shows that in 
\eqref{e87165386} and \eqref{e287211}  sometimes equality holds. 

\begin{corollary}
Let the assumptions of Lemma \ref{l2320843} hold. If 
$a,b,c,d\in\{0,1\}$, then 
$(u,v)\in\{(0,0),(1,0),(2,0),(2,1),(3,1),(4,2)\}$ and the 
values of $1-\conc{R}{0}$, $\sigma^2$, $1-\dtv{R}{\delta_1*R}$ 
and $w$ depending on $(u,v)$ are given in the following table: 
\begin{table}[H] \centering 
\begin{tabular}{c||c|c|c|c|c|c}
$(u,v)$ & $(0,0)$ & $(1,0)$ & $(2,0)$ & $(2,1)$ & $(3,1)$ & $(4,2)$ \\ 
\hline \hline

$1-\dtv{R}{\dirac_1*R}$ & $0$ & $1/2$ & $0$ & $0$ & $1/2$ & $0$ \\ 
\hline

$w$ & $0$ & $1/2$ & $1$ & $0$ & $1/2$ & $0$ \\  \hline

$1-\conc{R}{0}$ & $0$ & $1/2$ & $0$ & $1/2$ & $1/2$ & $0$ \\ \hline
$\sigma^2$ & $0$ & $1/4$ & $0$ & $1$ & $1/4$ & $0$ \\  

\end{tabular}    
\end{table}
\noindent
In particular, in \eqref{e87165386} equality holds. 
If $(u,v)\in\{(0,0),(2,0),(2,1),(4,2)\}$, then we have
$1-\conc{R}{0}=\frac{\sigma^2}{2}$;
if $(u,v)\in\{(0,0),(1,0),(2,0),(3,1),(4,2)\}$, then 
$1-\conc{R}{0}=2\sigma^2$.

\end{corollary}
\begin{proof}
This easily follows from Lemma \ref{l2320843}. 
\end{proof}
\begin{proof}[Proof of Theorem \ref{t98374587}]
Let us prove \eqref{e184376651} with the help of 
Theorem \ref{t9117372}. 
From \eqref{e87165386}, we obtain that, 
for $(j,k),(r,s)\in\set{n}_{\neq}^2$, 
\begin{align*}
\nu(j,k,r,s)
&:=1-\dtv{R(j,k,r,s)}{\dirac_1*R(j,k,r,s)}
\geq \tau(j,k,r,s)=:\widetilde{\nu}(j,k,r,s).
\end{align*}
Let $\varphi\in\set{n}_{\neq}^n$ be arbitrary and 
$\xi_q:=\xi_q(\varphi)$ 
for $q\in\{1,2\}$
be defined as in \eqref{e82663}. 
Since $\tau(j,k,r,s)\leq\varepsilon$ for all 
$(j,k),(r,s)\in\set{n}_{\neq}^2$, we obtain
$\xi_2-\xi_1^2\leq\varepsilon\xi_1$.
Therefore, \eqref{e224523} can be applied. 
Since $\varphi$ was chosen arbitrarily, we get
\begin{align*}
\dtv{S_n}{1+S_n}
&\leq \frac{C(\frac{1}{8\varepsilon})}{\sqrt{\pi}}
\frac{1}{\sqrt{\frac{1}{8}+
\floor{\frac{m}{2}}\frac{1}{n!}\sum_{\varphi\in\set{n}_{\neq}^n}
\xi_1(\varphi)}},
\end{align*}
where 
$\frac{1}{n!}\sum_{\varphi\in\set{n}_{\neq}^n}\xi_1(\varphi)
=\chi$. 
Inequality \eqref{e18437665} is similarly shown
using Theorem \ref{t22628362} with $t=0$ and \eqref{e287211}.  
\end{proof}


{\small 
\let\oldbibliography\thebibliography
\renewcommand{\thebibliography}[1]{\oldbibliography{#1}
\setlength{\itemsep}{0.8ex plus0.8ex minus0.8ex}} 
\linespread{1}
\selectfont
\bibliography{sirds_32}
}
\end{document}